\pdfoutput=1
\documentclass{article}
\usepackage[latin9]{inputenc}
\usepackage{color}
\usepackage{verbatim}
\usepackage{refstyle}
\usepackage{amsmath}
\usepackage{amsthm}
\usepackage{amssymb}
\usepackage{setspace}
\usepackage[unicode=true,
 bookmarks=false,
 breaklinks=false,pdfborder={0 0 1},backref=false,colorlinks=true]
 {hyperref}

\makeatletter

\AtBeginDocument{\providecommand\Exaref[1]{\ref{Exa:#1}}}
\AtBeginDocument{\providecommand\Thmref[1]{\ref{Thm:#1}}}
\AtBeginDocument{\providecommand\Lemref[1]{\ref{Lem:#1}}}
\AtBeginDocument{\providecommand\eqref[1]{\ref{eq:#1}}}
\AtBeginDocument{\providecommand\Queref[1]{\ref{Que:#1}}}
\AtBeginDocument{\providecommand\Corref[1]{\ref{Cor:#1}}}
\RS@ifundefined{subsecref}
  {\newref{subsec}{name = \RSsectxt}}
  {}
\RS@ifundefined{thmref}
  {\def\RSthmtxt{theorem~}\newref{thm}{name = \RSthmtxt}}
  {}
\RS@ifundefined{lemref}
  {\def\RSlemtxt{lemma~}\newref{lem}{name = \RSlemtxt}}
  {}

\theoremstyle{plain}
\newtheorem{thm}{\protect\theoremname}
\theoremstyle{plain}
\newtheorem{fact}[thm]{\protect\factname}
\theoremstyle{plain}
\newtheorem{question}[thm]{\protect\questionname}
\theoremstyle{definition}
\newtheorem{defn}[thm]{\protect\definitionname}
\theoremstyle{definition}
\newtheorem{example}[thm]{\protect\examplename}
\theoremstyle{plain}
\newtheorem{lem}[thm]{\protect\lemmaname}
\theoremstyle{remark}
\newtheorem{rem}[thm]{\protect\remarkname}
\theoremstyle{remark}
\newtheorem*{rem*}{\protect\remarkname}
\theoremstyle{plain}
\newtheorem{cor}[thm]{\protect\corollaryname}

\pdfoutput=1
\date{ }
\usepackage{graphicx}
\usepackage{amsfonts}
\usepackage{amsthm}
\usepackage{cases}
\usepackage{tikz-cd}
\allowdisplaybreaks

\DeclareMathOperator{\sgn}{sgn}
\usepackage{bbold}

\usepackage{refstyle}

\newref{cor}{name = corollary~, names = corollaries~, Name = Corollary~, Names = Corollaries~}
\newref{conj}{name= conjecture~, names = conjectures~, Name = Conjecture~, Names = Conjectures~}
\newref{exa}{name= example~, names = examples~, Name = Example~, Names = Examples~}
\newref{rem}{name= remark~, names = remarks~, Name = Remark~, Names = Remarks~} 
\newref{thm}{name= theorem~, names = theorems~, Name = Theorem~, Names = Theorems~}
\newref{lem}{name= lemma~, names = lemmas~, Name = Lemma~, Names = Lemmas~}
\newref{def}{name= definition~, names = definitions~, Name = Definition~, Names = Definitions~}
\newref{fact}{name= fact~, names = facts~, Name = Fact~, Names = Facts~}
\newref{blackbox}{name= blackbox~, names = blackboxes~, Name = Blackbox~, Names = Blackboxes~}
\newref{sec}{name= section~, names = sections~, Name = Section~, Names = Sections~}
\newref{subsec}{name= subsection~, names = subsections~, Name = Subsection~, Names = Subsections~}
\newref{que}{name= question~, names = questions~, Name = Question~, Names = Questions~}

\hypersetup{citecolor=purple,linkcolor=blue}

\AtBeginDocument{ 
\DeclareBibliographyAlias{article}{std}
\DeclareBibliographyAlias{book}{std}
\DeclareBibliographyAlias{booklet}{std}
\DeclareBibliographyAlias{collection}{std}
\DeclareBibliographyAlias{inbook}{std}
\DeclareBibliographyAlias{incollection}{std}
\DeclareBibliographyAlias{inproceedings}{std}
\DeclareBibliographyAlias{manual}{std}
\DeclareBibliographyAlias{misc}{std}
\DeclareBibliographyAlias{online}{std}
\DeclareBibliographyAlias{patent}{std}
\DeclareBibliographyAlias{periodical}{std}
\DeclareBibliographyAlias{proceedings}{std}
\DeclareBibliographyAlias{report}{std}
\DeclareBibliographyAlias{thesis}{std}
\DeclareBibliographyAlias{unpublished}{std}
\DeclareBibliographyAlias{*}{std}

\DeclareBibliographyDriver{std}{%
  \usebibmacro{bibindex}%
  \usebibmacro{begentry}%
  \usebibmacro{author/editor+others/translator+others}%
  \setunit{\labelnamepunct}\newblock
  \usebibmacro{title}%
  \newunit\newblock
  \printfield{journaltitle}
  \newunit\newblock
  \printfield{url}
  \newunit\newblock
  \usebibmacro{date}%
  \newunit\newblock
  \usebibmacro{finentry}}
}

\makeatother

\usepackage[style=alphabetic]{biblatex}
\providecommand{\corollaryname}{Corollary}
\providecommand{\definitionname}{Definition}
\providecommand{\examplename}{Example}
\providecommand{\factname}{Fact}
\providecommand{\lemmaname}{Lemma}
\providecommand{\questionname}{Question}
\providecommand{\remarkname}{Remark}
\providecommand{\theoremname}{Theorem}

\begin{document}
\title{A note on the Turing universality of homogeneous potential wells and
geodesible flows}
\author{Khang Manh Huynh}

\maketitle

\begin{abstract}
We explore some properties of flows with strongly adapted 1-forms,
originally discovered in \parencite{taoUniversalityPotentialWell2017},
which can be used to embed Turing machines into dynamical systems.
In particular, we discuss some relations to geodesible flows, and
show that even a slight modification of the dynamical system, such
as homogeneity, can lead to an intermediate class of flows between
adapted flows and geodesible flows, while still retaining Turing universality.
\end{abstract}

\subsection*{Acknowledgments}

The author is grateful to Terence Tao for valuable discussions during
the preparation of this work.

\section{Introduction}

In \parencite{taoFiniteTimeBlowup2016a}, Terence Tao demonstrated
the finite-time blowup of an averaged version of the Navier-Stokes
equation, by constructing a system of logic gates within ideal fluid.
This led to the idea of embedding an actual universal Turing machine
into dynamical systems. In \parencite{taoUniversalityPotentialWell2017},
it was demonstrated that a Turing machine, represented as the flow
of a smooth vector field, can be embedded into potential well and
nonlinear wave systems. Such systems are then considered to be universal,
with arbitrarily complex behaviors.

Since then, there have been various attempts to explore and classify
the types of flows and dynamical systems that could give rise to such
embeddings, as well as their relations to other subjects in differential
geometry. In particular, Tao showed that a flow generated by a non-vanishing
smooth vector field $X$ on a smooth, closed compact manifold $M$,
can be embedded into a potential well system if and and only if there
is\textbf{ }a 1-form $\theta$ such that 
\[
\theta\cdot X>0,\text{ and }\iota_{X}d\theta\text{ is exact}
\]
 We call $\theta$ a \textbf{strongly adapted $1$-form}. One very
important subclass of flows with strongly adapted 1-forms is that
of \textbf{geodesible flows }(i.e. there exists a Riemannian metric
$g$ that turns the orbits of $X$ into geodesics). By \parencite[Proposition 6.8, page 71]{tondeurGeometryFoliations1997}
or \parencite{gluckDynamicalBehaviorGeodesic1980}, a flow generated
by $X$ is geodesible if and only if there is\textbf{ }a 1-form $\theta$
such that 
\[
\theta\cdot X>0,\text{ and }\iota_{X}d\theta=0
\]

Note that we are using the definition of geodesibility from \parencite{tondeurGeometryFoliations1997},
which is somewhat weaker than that in \parencite{taoUniversalityPotentialWell2017},
where $g\left(X,X\right)$ is required to be $1$ and $\nabla_{X}X=0$.
We will call the latter\textbf{ strongly geodesible flows}. By \parencite[Proposition 6.7, page 71]{tondeurGeometryFoliations1997},
a flow generated by $X$ is strongly geodesible if and only if there
is a 1-form $\theta$ such that 
\[
\theta\cdot X=1,\text{ and }\iota_{X}d\theta=\mathcal{L}_{X}\theta=0
\]

A geodesible flow $X$ can be turned into a strongly geodesible flow
by multiplying $X$ with $\frac{1}{\theta\cdot X}$. We also note
the similarity to Reeb vector fields in contact geometry \parencite{etnyreContactTopologyHydrodynamics1997}.

All the examples from \parencite{taoUniversalityPotentialWell2017}
turn out to be geodesible, including the universal Turing machine
which is strongly geodesible. 

We observe that the class of flows with strongly adapted 1-forms is
distinct from the class of strongly geodesible flows (by \Exaref{torus}).
This suggests that the potential well system has more degrees of freedom
than necessary for Turing universality, and it should be possible
to put more constraints into the specifications of the system, such
as homogeneity, or embedding into a sphere, while still retaining
Turing universality. Indeed, that is what this note aims to demonstrate
via the Nash embedding theorem.
\begin{fact}
The class of potential well systems with homogeneous potentials of
degree $k$(where $k\neq0$) is universal.
\end{fact}

As an aside, we also discover another class of flows characterizing
homogeneous potential wells of degree two (in \Thmref{homo2con} and
\Thmref{WellHomo}), that is distinct from the class of flows with
strongly adapted 1-forms (see \Exaref{torus}). 
\begin{question}
\label{que:how_many} How many distinct classes of flows characterizing
dynamical systems exist between strongly geodesible flows and flows
with strongly adapted 1-forms? Can they be usefully classified?
\end{question}

Since the preparation of this note, there have been various new developments,
including the investigation of the Euler equation in \parencite{taoUniversalityIncompressibleEuler2019,taoUniversalityIncompressibleEuler2018,peralta-salasCharacterization3DSteady2020}.
In particular, \parencite{cardonaUniversalityEulerFlows2019} demonstrated
the Turing universality of stationary Euler flows, using tools from
contact topology and geometric $h$-principles. 

\section{Preliminary definitions\label{sec:Preliminary-definitions}}
\begin{defn}
Let $M$ be a smooth manifold and $X\in\mathfrak{X}M$ be a smooth
vector field of $M$. Then $(M,X)$ is called a\emph{ (smooth) flow}.
The flow is called \emph{compact} if $M$ is compact, and \emph{nonsingular}
if $X$ is nonsingular $(X_{p}\neq0\;\forall p\in M)$. 
\end{defn}

\begin{example}
We consider two basic systems:
\begin{itemize}
\item When $M=T^{*}\mathbb{R}^{n}=\mathbb{R}^{n}\times\mathbb{R}^{n}$ and
$X(q,p)=(p,-\nabla V(q))$ where $V:\mathbb{R}^{n}\to\mathbb{R}$
is a smooth potential function, we denote the flow as $\text{Well}(\mathbb{R}^{n},V)$.
The associated Hamiltonian is $\mathfrak{H}(q,p)=\frac{|p|^{2}}{2}+V(q)$,
while the Lagrangian is $\mathfrak{L}(q,p)=\frac{|p|^{2}}{2}-V(q).$
The symplectic potential is defined in coordinates as $\Theta=p_{i}dq^{i}$,
while the canonical symplectic form is $\Omega=-d\Theta=dq^{i}\wedge dp_{i}$.
Then $\Theta\cdot X=\left|p^{2}\right|$ and 
\[
\iota_{X}d\Theta=-d\mathfrak{H},\mathcal{L}_{X}\Theta=d\mathcal{L}
\]
\item Let $\mathbb{T}=\mathbb{R}/\mathbb{Z}$ and $H_{d}^{m}=C^{\infty}(\mathbb{T}^{d}\to R^{m})$.
When $M=T^{*}(H_{d}^{m})=H_{d}^{m}\times H_{d}^{m}$ (a Frechet manifold/vector
space) and $X(q,p)=(p,\Delta_{\mathbb{T}^{d}}q-\nabla_{\mathbb{R}^{m}}V(q))$
where $V:\mathbb{R}^{m}\to\mathbb{R}$ is a smooth potential function,
we denote the flow as $\mathrm{NLW}(\mathbb{T}^{d},\mathbb{R}^{m},V)$. 
\end{itemize}
We can treat $\text{Well}(\mathbb{R}^{m},V)$ as a subset of $\mathrm{NLW}(\mathbb{T}^{d},\mathbb{R}^{m},V)$
where $q,p$ are constant functions.
\end{example}

\begin{defn}
An \emph{embedding} of the flow $(N,Y)$ into the flow $(M,X)$ is
an embedding $\phi:N\to M$ such that $d\phi\cdot Y=X$. When $M$
is $T^{*}(H_{d}^{m})$, an embedding means a smooth (in a Gateaux
sense) injective immersion. 

A\emph{ }1-form\emph{ $\theta$ }of the flow $(M,X)$ is called \textbf{weakly
adapted }when $\theta\cdot X\geq0$ and $\mathcal{L}_{X}(\theta)$
is exact. $\theta$ is called \textbf{strongly adapted} if $\theta\cdot X>0$.
\end{defn}

\begin{fact}
By \cite{taoUniversalityPotentialWell2017}, there is an embedding
from $(M,X)$ into $\text{Well}(\mathbb{R}^{n},V)$ for some $n$
and $V$ if and only if there is a strongly adapted 1-form for $(M,X).$
Under this embedding, $\theta\cdot X$ corresponds to the kinetic
energy $|p|^{2}$ and $\mathcal{L}_{X}(\theta)$ corresponds to $d\mathfrak{L}$
where $\mathfrak{L}$ is the Lagrangian.
\end{fact}

\begin{defn}
If the flow $\text{Well}(\mathbb{R}^{n},V)$ further satisfies that
$V$ is homogeneous of order $k$ (away from the origin), i.e. 
\[
V(tx)=t^{k}V(x)\;\forall t\geq1,\forall x\in\mathbb{R}^{n}\setminus B(0,\epsilon)
\]
 for some $\epsilon\in(0,1)$, then we denote the flow as $\text{HomWell}_{k}(\mathbb{R}^{n},V)$. 

An embedding $\phi:N\to\mathbb{R}^{n}\times\mathbb{R}^{n},a\mapsto(Q(a),P(a))$
from $(N,Y)$ into $\text{HomWell}_{k}(\mathbb{R}^{n},V)$ is understood
to satisfy $|Q|>0$ and $\mathrm{Im}Q$ lies in the region where $V$
is homogeneous. The embedding is also called \emph{spherical} if $Q(N)\subseteq cS^{n-1}$
for some $c>0$.

Similarly for $\mathrm{NLW}(\mathbb{T}^{d},\mathbb{R}^{m},V)$, if
$V$ is homogeneous of order $k$, we denote it as $\mathrm{HomNLW_{k}}(\mathbb{T}^{d},\mathbb{R}^{m},V)$. 
\end{defn}

\begin{fact}[Euler's homogeneous function theorem]
 A smooth potential $V$ is homogeneous of order $k$ away from the
origin if and only if 
\begin{equation}
\left\langle \nabla V(x),x\right\rangle =kV(x)\;\forall x\in\mathbb{R}^{n}\setminus B(0,\epsilon)\label{eq:euler_Cond}
\end{equation}
 for some $\epsilon>0$ .
\end{fact}

\section{Embedding into HomNLW}

First, we need a technical lemma to characterize the Euler condition
in our setting:
\begin{lem}
\label{lem:1st_tech_lem}Let $(N,Y)$ be a compact nonsingular flow
and $k\in\mathbb{R}$. Let $Y$ be extended to $(Y,0)\in\mathfrak{X}(N\times\mathbb{T}^{d})$.

Let $Q:N\times\mathbb{T}^{d}\to\mathbb{R}^{m}$ be smooth, and $\theta\in\Omega^{1}\left(N\times\mathbb{T}^{d}\right)$
be such that $\theta\cdot W=\left\langle YQ,WQ\right\rangle $ $\forall$$W\in\mathfrak{X}(N\times\mathbb{T}^{d})$.

Then $\forall W\in\mathfrak{X}(N\times\mathbb{T}^{d})$:
\begin{align}
 & W\left\langle -YYQ+\Delta_{x}Q,Q\right\rangle =k\left\langle -YYQ+\Delta_{x}Q,WQ\right\rangle \label{eq:first_crit}\\
\iff & k\mathcal{L}_{Y}\theta\cdot W+W\left(\left(1-\frac{k}{2}\right)\left(\theta\cdot Y\right)-YY\left(\frac{|Q|^{2}}{2}\right)-\left(1-\frac{k}{2}\right)\left|\nabla_{x}Q\right|^{2}+\Delta_{x}\left(\frac{|Q|^{2}}{2}\right)\right)\nonumber \\
 & \;\;\;\;\;=k\sum_{i}\partial_{x^{i}}\left\langle \partial_{x^{i}}Q,WQ\right\rangle -k\sum_{i}\left\langle \partial_{x^{i}}Q,\left[\partial_{x_{i}},W\right]Q\right\rangle \label{eq:have_to_satisfy}
\end{align}
\end{lem}

\begin{proof}
Observe that:
\begin{align*}
(\ref{eq:first_crit})\iff & W\left(\left|YQ\right|^{2}-Y\left\langle YQ,Q\right\rangle -\left|\nabla_{x}Q\right|^{2}+\sum_{i}\partial_{x^{i}}\left\langle Q,\partial_{x^{i}}Q\right\rangle \right)\\
 & \;\;\;\;\;=k\left(-Y\left\langle YQ,WQ\right\rangle +\left\langle YQ,[Y,W]Q\right\rangle +W\left(\frac{|YQ|^{2}}{2}\right)+\sum_{i}\partial_{x^{i}}\left\langle \partial_{x^{i}}Q,WQ\right\rangle \right.\\
 & \;\;\;\;\;\;\;\;\;\;\left.-\sum_{i}\left\langle \partial_{x^{i}}Q,\left[\partial_{x_{i}},W\right]Q\right\rangle -W\left(\frac{\left|\nabla_{x}Q\right|^{2}}{2}\right)\right)\\
\iff & W\left(\left(1-\frac{k}{2}\right)\left|YQ\right|^{2}-YY\left(\frac{|Q|^{2}}{2}\right)-\left(1-\frac{k}{2}\right)\left|\nabla_{x}Q\right|^{2}+\Delta_{x}\left(\frac{|Q|^{2}}{2}\right)\right)\\
 & \;\;\;\;\;=-kY(\theta\cdot W)+k\theta\cdot\left[Y,W\right]+k\left(\sum_{i}\partial_{x^{i}}\left\langle \partial_{x^{i}}Q,WQ\right\rangle -\left\langle \partial_{x^{i}}Q,\left[\partial_{x_{i}},W\right]Q\right\rangle \right)
\end{align*}
We are done as $Y(\theta\cdot W)=\mathcal{L}_{Y}\theta\cdot W+\theta\cdot\left[Y,W\right]$.
\end{proof}
\begin{rem}
To see how this lemma is used, let us assume there is an embedding
$y\mapsto\left(Q\left(y,\cdot\right),P\left(y,\cdot\right)\right)$
from $(N,Y)$ into $\mathrm{HomNLW_{k}}(\mathbb{T}^{d},\mathbb{R}^{m},V)$
for some $V$ and some $m$ (with $P=YQ,$ $YP=\Delta_{x}Q-\nabla_{\mathbb{R}^{m}}V(Q)$).
Then define $\widetilde{\theta}\in\Omega^{1}\left(N\times\mathbb{T}^{d}\right)$
and $\theta\in\Omega^{1}\left(N\right)$ by
\begin{align*}
\widetilde{\theta}\cdot W & =\left\langle YQ,WQ\right\rangle \;\forall W\in\mathfrak{X}\left(N\times\mathbb{T}^{d}\right),\\
\theta\cdot Z & =\int_{\mathbb{T}^{d}}\left\langle YQ,ZQ\right\rangle \;\mathrm{d}x\;\forall Z\in\mathfrak{X}\left(N\right)
\end{align*}
 We note that we have implicity extended $Z$ to $\left(Z,0\right)\in\mathfrak{X}\left(N\times\mathbb{T}^{d}\right)$.
This implies 
\begin{align}
\theta\cdot Z & =\int_{\mathbb{T}^{d}}\widetilde{\theta}\cdot Z\;\mathrm{d}x\;\forall Z\in\mathfrak{X}\left(N\right),\nonumber \\
\mathcal{L}_{Y}\theta\cdot Z & =Y\left(\theta\cdot Z\right)-\theta\cdot\left[Y,Z\right]\nonumber \\
 & =\int_{\mathbb{T}^{d}}Y\left(\widetilde{\theta}\cdot Z\right)-\widetilde{\theta}\cdot\left[Y,Z\right]\;\mathrm{d}x\nonumber \\
 & =\int_{\mathbb{T}^{d}}\mathcal{L}_{Y}\widetilde{\theta}\cdot Z\;\mathrm{d}x\label{eq:LYZ}
\end{align}
On the other hand, we have 
\begin{align*}
\mathcal{L}_{Y}\theta\cdot Z & =\int_{\mathbb{T}^{d}}Y\left\langle YQ,ZQ\right\rangle -\left\langle YQ,\left[Y,Z\right]Q\right\rangle \;\mathrm{d}x\\
 & =\int_{\mathbb{T}^{d}}\left\langle YYQ,ZQ\right\rangle +Z\left(\frac{1}{2}\left|YQ\right|^{2}\right)\;\mathrm{d}x\\
 & =\int_{\mathbb{T}^{d}}\left\langle \Delta_{x}Q-\left(\nabla_{\mathbb{R}^{m}}V\right)\circ Q,ZQ\right\rangle +Z\left(\frac{1}{2}\left|YQ\right|^{2}\right)\;\mathrm{d}x\\
 & =Z\int_{\mathbb{T}^{d}}\left(-\frac{1}{2}\left|\nabla_{x}Q\right|^{2}-V\circ Q+\frac{1}{2}\left|YQ\right|^{2}\right)\;\mathrm{d}x
\end{align*}
where we have used the fact that $\left[\partial_{x^{i}},Z\right]=0$
as $Z\in\mathfrak{X}\left(N\right)$. Then there is $L$ smooth on
$N$ such that $\mathcal{L}_{Y}\theta=dL$.

From (\ref{eq:euler_Cond}), for any $W\in\mathfrak{X}(N\times\mathbb{T}^{d})$:
\begin{align}
 & W\left\langle \left(\nabla_{\mathbb{R}^{m}}V\right)\circ Q,Q\right\rangle =W\left(kV\circ Q\right)\nonumber \\
\iff & W\left\langle \left(\nabla_{\mathbb{R}^{m}}V\right)\circ Q,Q\right\rangle =k\left\langle \left(\nabla_{\mathbb{R}^{m}}V\right)\circ Q,WQ\right\rangle \nonumber \\
\iff & W\left\langle -YYQ+\Delta_{x}Q,Q\right\rangle =k\left\langle -YYQ+\Delta_{x}Q,WQ\right\rangle \nonumber \\
\iff & k\mathcal{L}_{Y}\widetilde{\theta}\cdot W+W\left(\left(1-\frac{k}{2}\right)\left(\widetilde{\theta}\cdot Y\right)-YY\left(\frac{|Q|^{2}}{2}\right)-\left(1-\frac{k}{2}\right)\left|\nabla_{x}Q\right|^{2}+\Delta_{x}\left(\frac{|Q|^{2}}{2}\right)\right)\nonumber \\
 & \;\;\;\;\;=k\sum_{i}\partial_{x^{i}}\left\langle \partial_{x^{i}}Q,WQ\right\rangle -k\sum_{i}\left\langle \partial_{x^{i}}Q,\left[\partial_{x_{i}},W\right]Q\right\rangle \label{eq:euler_imply}
\end{align}
where we have used \Lemref{1st_tech_lem} to pass to the last equality.

Let $N$ be locally parametrized by functions $(y^{j})$. Letting
$W=\partial_{y^{j}}$ locally (therefore a smooth vector field on
an open set of $N$), by (\ref{eq:LYZ}), we have
\[
\int_{\mathbb{T}^{d}}\mathcal{L}_{Y}\widetilde{\theta}\cdot W\;\mathrm{d}x=\mathcal{L}_{Y}\theta\cdot W=WL
\]
 Then, with $W=\partial_{y^{j}}$, by integrating in $x$, (\ref{eq:euler_imply})
implies

\begin{equation}
kL+\left(1-\frac{k}{2}\right)\left(\theta\cdot Y\right)-YY\left(\int_{\mathbb{T}^{d}}\frac{|Q|^{2}}{2}\;\mathrm{d}x\right)-\left(1-\frac{k}{2}\right)\left(\int_{\mathbb{T}^{d}}\left|\nabla_{x}Q\right|^{2}\mathrm{d}x\right)=\text{const on }N\label{eq:necessary}
\end{equation}
\end{rem}

This is the main obstruction to our embedding. 

It turns out that when $k\notin\{0,2\}$, it can be trivially satisfied,
and any compact nonsingular flow that can be embedded into $\mathrm{NLW}$
can also be embedded into $\mathrm{HomNLW}_{k}$.

\begin{thm}[Homogeneous NLW embedding]
\label{thm:NLW_embed} Let $k\notin\{0,2\},\;d\in\mathbb{N}$. Let
$(N,Y)$ be a compact nonsingular flow with a strongly adapted 1-form
$\theta$. Then there is an embedding from $(N,Y)$ into $\mathrm{HomNLW_{k}}(\mathbb{T}^{d},\mathbb{R}^{m},V)$
for some $V$ and some $m$. 
\end{thm}

\begin{proof}
As in \cite{taoUniversalityPotentialWell2017}, since $\theta\cdot Y>0$,
we can construct a Riemannian metric $g$ on $N$ such that $g\cdot Y=\theta$.
Let $L$ be smooth on $N$ such that $dL=\mathcal{L}_{Y}(\theta)$. 

Any smooth function $f$ on $N$ can be extended to a smooth function
$f(y,x)=f(y)$ on $N\times\mathbb{T}^{d}$. Any vector field $Z\in\mathfrak{X}N$
can be extended to $(Z,0)\in\mathfrak{X}(N\times\mathbb{T}^{d})$.
Similarly for 1-forms. So under these extensions, we can say $Y\in\mathfrak{X}(N\times\mathbb{T}^{d})$,
$\theta\in\Omega^{1}(N\times\mathbb{T}^{d})$ and $dL=\mathcal{L}_{Y}(\theta)$.
Let $N$ be locally parametrized by functions $(y^{j})$, and $\mathbb{T}^{d}$
globally parametrized by $(x^{i})$. Since $Y$ is nonvanishing, by
straightening, WLOG assume $[Y,\partial_{y^{j}}]=0\;\forall j.$ 

We wish to extend $g$ to a Riemannian metric on $N\times\mathbb{T}^{d}$
and satisfy (\ref{eq:have_to_satisfy}) later.

As the first step, we let $R$ be any smooth function on $N$ such
that $R>0$ (then extend to $R(y,x)=R(y)$). Because $k\neq2$, we
can define $G$ on $N\times\mathbb{T}^{d}$ such that 
\begin{equation}
kL+\left(1-\frac{k}{2}\right)\left(\theta\cdot Y\right)-YY\left(\frac{R^{2}}{2}\right)-\left(1-\frac{k}{2}\right)G+\Delta_{x}\left(\frac{R^{2}}{2}\right)=\text{const}\label{eq:EulersetG}
\end{equation}
where the constant is chosen so that $G>0.$ Then $G$ is constant
in $x$ and we define $g(\partial_{y^{j}},\partial_{x^{i}})=0$, $g(\partial_{x^{i}},\partial_{x^{j}})=\delta_{ij}\frac{G}{d}$.
So 
\[
\forall W\in\mathfrak{X}(N\times\mathbb{T}^{d}):g(\partial_{x^{i}},W)=dx^{i}(W)\frac{G}{d}.
\]

We observe that $G=\sum_{i=1}^{d}g(\partial_{x^{i}},\partial_{x^{i}})$
and $g(\partial_{x^{i}},\partial_{x^{i}})>0\;\forall i=\overline{1,d}$.
Then $g$ is now a Riemannian metric on $N\times\mathbb{T}^{d}$ and
we still have $g\cdot Y=\theta$. We can also check that, if $W$
is $\partial_{y^{j}}$ or $\partial_{x^{j}}$:
\begin{align}
 & W\left(kL+\left(1-\frac{k}{2}\right)\left(\theta\cdot Y\right)-YY\left(\frac{R^{2}}{2}\right)-\left(1-\frac{k}{2}\right)G+\Delta_{x}\left(\frac{R^{2}}{2}\right)\right)\label{eq:okay_step}\\
 & \;\;\;\;\;=\sum_{i=1}^{d}k\partial_{x^{i}}g\left(\partial_{x^{i},}W\right)=0\nonumber 
\end{align}

Our goal is to have $Q(y,x)=R(y,x)S(y,x)$ where $R^{2}$ is sufficiently
large and $S:N\times\mathbb{T}^{d}\to S^{m-1}$ is a map given by
some Nash embedding. 

To this end, we introduce another Riemannian metric on $N\times\mathbb{T}^{d}$:
let $h_{\alpha\beta}=\frac{g_{\alpha\beta}-(\partial_{\alpha}R)(\partial_{\beta}R)}{R^{2}}$.
Then $h$ is positive definite iff $\left(g_{\alpha\beta}-(\partial_{\alpha}R)(\partial_{\beta}R)\right)_{\alpha,\beta}>0$.
We now observe that $g_{\alpha\beta}-(\partial_{\alpha}R)(\partial_{\beta}R)=g_{\alpha\beta}-\frac{(\partial_{\alpha}R^{2})(\partial_{\beta}R^{2})}{4R^{2}}$.
As we can add a positive constant to $R^{2}$ and $g>0$, WLOG $h$
is a Riemannian metric on $N\times\mathbb{T}^{d}$. 

Then by Nash embedding, and the fact that any compact regular submanifold
of $\mathbb{R}^{l}$ (with the usual Euclidean metric) can be embedded
into $cS^{2l}$ for some $c>0$,%
\begin{comment}
\footnote{See, for instance, \href{https://math.stackexchange.com/questions/1790879/can-every-riemannian-manifold-be-embedded-in-a-sphere}{https://math.stackexchange.com/questions/1790879/can-every-riemannian-manifold-be-embedded-in-a-sphere}} 
\end{comment}
{} we can rescale $g,h$ and $\theta$ to get an isometric embedding
$S:(N\times\mathbb{T}^{d},h)\to S^{m-1}$ for some $m$. 

Then $\left\langle \partial_{\alpha}S,\partial_{\beta}S\right\rangle =h_{\alpha\beta}=\frac{g_{\alpha\beta}-(\partial_{\alpha}R)(\partial_{\beta}R)}{R^{2}}$. 

Let $Q=RS$. Then $\left\langle \partial_{\alpha}Q,\partial_{\beta}Q\right\rangle =\left(\partial_{\alpha}R\right)\left(\partial_{\beta}R\right)+R^{2}\left\langle \partial_{\alpha}S,\partial_{\beta}S\right\rangle =g_{\alpha\beta}$
and $Q:(N\times\mathbb{T}^{d},g)\to\mathbb{R}^{m}$ is also an isometric
embedding. This means for any $W\in\mathfrak{X}(N\times\mathbb{T}^{d})$
we have $\theta(W)=g(Y,W)=\left\langle YQ,WQ\right\rangle $.

Let $W$ be $\partial_{y^{j}}$ or $\partial_{x^{j}}$ for some $j,$
then $W,Y,\partial_{x^{i}}$ commute. By (\ref{eq:okay_step}) and
\Lemref{1st_tech_lem} we have:
\begin{align}
 & W\left(kL+\left(1-\frac{k}{2}\right)\left(\theta\cdot Y\right)-YY\left(\frac{|Q|^{2}}{2}\right)-\left(1-\frac{k}{2}\right)\left|\nabla_{x}Q\right|^{2}+\Delta_{x}\left(\frac{|Q|^{2}}{2}\right)\right)\nonumber \\
 & \;\;\;\;\;=k\sum_{i}\partial_{x^{i}}\left\langle \partial_{x^{i}}Q,WQ\right\rangle \nonumber \\
\iff & W\left\langle -YYQ+\Delta_{x}Q,Q\right\rangle =k\left\langle -YYQ+\Delta_{x}Q,WQ\right\rangle \label{eq:deriveeuler}
\end{align}

By $C^{\infty}(N\times\mathbb{T}^{d})$-linearity, we conclude that
(\ref{eq:deriveeuler}) holds true for any $W\in\mathfrak{X}(N\times\mathbb{T}^{d}).$ 

Now we construct the potential. We define $v$ on $N\times\mathbb{T}^{d}$
such that 
\begin{equation}
kv=\left\langle -YYQ+\Delta_{x}Q,Q\right\rangle \label{eq:small_v}
\end{equation}
 Note that this is where we need $k\neq0$. 

Let $V_{0}=v\circ Q^{-1}$ be the restricted potential on $\mathrm{Im}Q$.
Then (\ref{eq:deriveeuler}) gives 
\begin{equation}
k\left\langle -YYQ+\Delta_{x}Q,WQ\right\rangle =Wkv=W(kV_{0}\circ Q)=k\left\langle \nabla_{\mathrm{Im}Q}V_{0}(Q),WQ\right\rangle \label{eq:arbit_W}
\end{equation}

Let us define $A=-YYQ+\Delta_{x}Q$. As $W$ is arbitrary in (\ref{eq:arbit_W}),
we conclude $\mathrm{proj}_{T(\mathrm{Im}Q)}A=\nabla_{\mathrm{Im}Q}V_{0}(Q)$,
while (\ref{eq:small_v}) implies $\left\langle A,Q\right\rangle =kV_{0}(Q)$.
To finally recover the Euler condition, we hope to extend $V_{0}$
to $V$ such that $\nabla_{\mathbb{R}^{m}}V(Q)=A$. We first work
on the unit sphere.

Homogeneity suggests we define $V_{1}$ on $\mathrm{Im}S$ as $V_{1}(S)=\frac{1}{R^{k}}V_{0}(RS)$.
Then we have 
\begin{eqnarray*}
W(V_{1}\circ S) & = & W\left(\frac{1}{R^{k}}V_{0}\circ Q\right)=\frac{-k}{R^{k+1}}(WR)(V_{0}\circ Q)+\frac{1}{R^{k}}W(V_{0}\circ Q)\\
 & = & \frac{1}{R^{k-1}}\left(\frac{-k}{R^{2}}(WR)(V_{0}\circ Q)+\frac{1}{R}\left\langle A,WQ\right\rangle \right)\\
 & = & \frac{1}{R^{k-1}}\left(\frac{-(WR)}{R^{2}}\left\langle A,RS\right\rangle +\frac{1}{R}\left\langle A,(WR)S\right\rangle +\frac{1}{R}\left\langle A,R(WS)\right\rangle \right)\\
 & = & \left\langle \frac{A}{R^{k-1}},WS\right\rangle 
\end{eqnarray*}
This means $\left\langle \nabla_{\mathrm{Im}S}V_{1}(S),WS\right\rangle =\left\langle \frac{A}{R^{k-1}},WS\right\rangle $.
Write $B=\frac{A}{R^{k-1}}$. Then $\mathrm{proj}_{T(\mathrm{Im}S)}B=\nabla_{\mathrm{Im}S}V_{1}(S)$
and $\left\langle B,S\right\rangle =\left\langle \frac{A}{R^{k-1}},\frac{Q}{R}\right\rangle =\frac{k}{R^{k}}V_{0}(Q)=kV_{1}(S)$.

We locally parametrize $\mathrm{Im}S$, $S^{m-1}$ and $\mathbb{R}^{m}$
in a neighborhood $U$ by $(a^{i}),(a^{i},b^{j})$ and $(a^{i},b^{j},c)$
respectively since $\mathrm{Im}S\hookrightarrow S^{m-1}\hookrightarrow\mathbb{R}^{m}.$
Let $a=(a^{i}),b=(b^{j})$. WLOG assume $\{b=0,c=1\}=U\cap\mathrm{Im}S,$
$\{c=1\}=U\cap S^{m-1}$. Then in local coordinates:
\[
B(a)=\partial_{a}V_{1}(a)\partial_{a}+F(a)\partial_{b}+kV_{1}(a)\partial_{c}
\]
where $F\in C^{\infty}(\mathrm{Im}S\cap U)$. Then define $V_{1}$
on $S^{m-1}\cap U$: 
\[
V_{1}(a,b)=V_{1}(a)+\left\langle F(a),b\right\rangle 
\]
 Then $\partial_{a}V_{1}(a,0)=\partial_{a}V_{1}(a)$ and $\partial_{b}V_{1}(a,0)=F(a)$
so $\mathrm{proj}_{T(S^{m-1})}B=\nabla_{S^{m-1}}V_{1}(S)$ on $S^{-1}(U)$.
From this local result, by partition of unity, we can extend $V_{1}$
to all of $S^{m-1}$ to have $V_{1}$ smooth on $S^{m-1}$ and $\mathrm{proj}_{T(S^{m-1})}B=\nabla_{S^{m-1}}V_{1}(S)$. 

Then by homogeneity we get $V$ on $\mathbb{R}^{m}$ (except possibly
a small neighborhood near the origin). Because $\left\langle B,S\right\rangle =kV(S)=\left\langle \nabla_{\mathbb{R}^{m}}V,S\right\rangle $
and $\mathrm{proj}_{T(S^{m-1})}B=\nabla_{S^{m-1}}V_{1}(S)=\mathrm{proj}_{T(S^{m-1})}\nabla_{\mathbb{R}^{m}}V(S)$,
we conclude 
\[
B=\nabla_{\mathbb{R}^{m}}V(S)
\]
Homogeneity implies that 
\[
\nabla_{\mathbb{R}^{m}}V(Q)=\nabla_{\mathbb{R}^{m}}V(RS)=R^{k-1}\nabla_{\mathbb{R}^{m}}V(S)=R^{k-1}B=A
\]
 and we are done. 
\end{proof}
\begin{rem*}
This roughly means that, from the viewpoint of complexity, bounded
orbits of $\mathrm{NLW}$ and $\mathrm{HomNLW}_{k}$ ($k\notin\{0,2\}$)
look alike.

From the proof, we in fact have a stronger conclusion: given any function
$R$ smooth on $N\times\mathbb{T}^{d}$ such that $\partial_{x}R=0$
and $R>0,$ there is $V$ on $\mathbb{R}^{m}$ homogeneous of order
$k$ and an embedding $\hat{\phi}:N\times\mathbb{T}^{d}\to\mathbb{R}^{m}\times\mathbb{R}^{m},\;(y,x)\mapsto(Q(y,x),P(y,x))$
such that
\[
|Q|^{2}-R^{2}=\text{constant},YQ=P,\;YP=\Delta_{x}Q-\nabla_{\mathbb{R}^{m}}V(Q)
\]
\end{rem*}
\begin{rem}
\label{rem:casek=00003D2}When $k=2$, (\ref{eq:necessary}) implies
$2L-YY\left(\int_{\mathbb{T}^{d}}\frac{|Q|^{2}}{2}\right)$ is constant
on $N$. Since it is understood that $|Q|>0$, we have the following
theorem.
\end{rem}

\begin{thm}
\label{thm:homo2con}Let $d\in\mathbb{N}$ and $(N,Y)$ be a compact
nonsingular flow. There is an embedding from $(N,Y)$ into $\mathrm{HomNLW_{2}}(\mathbb{T}^{d},\mathbb{R}^{m},V)$
for some $m$ and $V$ iff there is a function $R>0$ smooth on $N$
and a strongly adapted 1-form $\theta$ such that 
\begin{equation}
d\left(YY\left(R^{2}\right)\right)-4\mathcal{L}_{Y}\theta=0\label{eq:new-condition}
\end{equation}
\end{thm}

\begin{proof}
Necessity is obvious from (\ref{eq:necessary}).

To prove sufficiency, we trivially extend $R$ to $N\times\mathbb{T}^{d}$
by $R(y,x):=R(y)$, and set $G$ = 1. Then we recover \eqref{EulersetG}.
\end{proof}
\begin{rem}
(\ref{eq:new-condition}) would be satisfied if we could find a $Y$-invariant
strongly adapted 1-form $\tilde{\theta}$ (i.e. $\mathcal{L}_{Y}\tilde{\theta}=0$).
Certainly, if $Y$ is strongly geodesible or isometric (Killing vector
field), we can let $\widetilde{\theta}$ be $Y^{\flat}$. The natural
question to ask is whether (\ref{eq:new-condition}) is superfluous,
once given the existence of strongly adapted 1-forms. The answer is
no, due to the following example.
\end{rem}

\begin{example}
\label{exa:torus}Let $N=\mathbb{T}^{2}$ and $Y=f(y)\partial_{x}$
where $f\in C^{\infty}(\mathbb{T})$. For instance: $f(y)=\sin(2\pi y)$+2.
Then $f$ never vanishes and $\mathrm{\sgn}f$ does not change. Because
we can replace $Y$ by $-Y$ and $\theta$ by $-\theta$, WLOG we
assume $f>0.$

Then we note that $\mathcal{L}_{Y}(dx)=d(\iota_{Y}(dx))=df=\partial_{y}f\;dy$
and $\mathcal{L}_{Y}(dy)=d(\iota_{Y}(dy))=0$. 

Let $\theta=\theta_{1}dx+\theta_{2}dy$. Then $\mathcal{L}_{Y}\theta=f\partial_{x}\theta_{1}dx+\left(\theta_{1}\partial_{y}f+f\partial_{x}\theta_{2}\right)dy$.
We note that a 1-form is exact on $\mathbb{T}^{2}$ iff it is closed
and vanishes under integration in the $x$-direction and $y$-direction
(the generators of the fundamental group).

So $\theta$ is strongly adapted if and only if
\[
\begin{cases}
\partial_{y}\left(f\partial_{x}\theta_{1}\right)-\partial_{x}\left(\theta_{1}\partial_{y}f+f\partial_{x}\theta_{2}\right)=0\\
\int_{\mathbb{T}}f(y)\partial_{x}\theta_{1}(x,y)\;\mathrm{d}x=0 & (\text{always true})\\
\int_{\mathbb{T}}\theta_{1}(x,y)\partial_{y}f(y)+f(y)\partial_{x}\theta_{2}(x,y)\;\mathrm{d}y=0\\
f\theta_{1}>0
\end{cases}
\]

The last condition just means $\theta_{1}>0$.

An easy pick is $\theta_{1}=1,\;\theta_{2}=\text{const}$. Or $\theta_{1}(x,y)=f(y),\;\theta_{2}=\text{const. Or }$
$\theta_{1}=\frac{1}{f},$ $\;\theta_{2}=\text{const}.$ So strongly
adapted 1-forms exist here and there are many kinds of them.

We ask whether there is a strongly adapted $\theta$ such that $\mathcal{L}_{Y}\theta=d(YYr)$
for some $r\in C^{\infty}(N)$. Since 
\[
d(YYr)=d\left(f^{2}\partial_{xx}r\right)=f^{2}\partial_{xxx}rdx+\left(2f\partial_{y}f\partial_{xx}r+f^{2}\partial_{xxy}r\right)dy
\]
this would mean 
\[
\begin{cases}
f\partial_{xxx}r-\partial_{x}\theta_{1}=0\\
2f\partial_{y}f\partial_{xx}r+f^{2}\partial_{xxy}r-\theta_{1}\partial_{y}f-f\partial_{x}\theta_{2}=0
\end{cases}
\]

But by integrating the second equation in $x$, we get $\partial_{y}f(y)\int_{\mathbb{T}}\theta_{1}(x,y)dx=0\;\forall y$.
If there is $y_{0}$ such that $f'(y_{0})\neq0$, such as when $f(y)=\sin(2\pi y)$+2,
then $\int_{\mathbb{T}}\theta_{1}(x,y_{0})dx=0$. But $\theta_{1}>0$
so this is a contradiction. Therefore the condition in \Thmref{homo2con},
as well as the existence of a $Y$-invariant strongly adapted 1-form,
is not superfluous once given the existence of strongly adapted 1-forms.
\end{example}

\begin{rem}
Effectively, it means there are flows that can embed into $\mathrm{NLW}(\mathbb{T}^{d},\mathbb{R}^{m},V)$
but not $\mathrm{HomNLW_{2}}(\mathbb{T}^{d},\mathbb{R}^{m},V)$, which
suggests the case $k=2$ is much more different than the other cases. 
\end{rem}

\Exaref{torus} also provides a simple proof of the following fact:
\begin{fact}
The class of nonsingular flows with strongly adapted 1-forms is strictly
bigger than the class of strongly geodesible flows, which itself includes
the universal Turing machine.
\end{fact}

If we restrict our attention to just volume-preserving vector fields,
\parencite{cieliebakNoteStationaryEuler2017} has also shown a similar
result via a counterexample with stationary Euler flows.

\section{Embedding into HomWell}

The case of embedding $(N,Y)$ into $\mathrm{HomWell}_{k}(\mathbb{R}^{m},V)$
can be thought of as a special case where $Q:N\to H_{d}^{m}$ maps
points on $N$ to constant functions, or when $d=0$. As $\left|\nabla_{x}Q\right|$
becomes zero, applying the exterior derivative to (\ref{eq:necessary})
gives us the following theorem.
\begin{thm}
\label{thm:WellHomo}Given $k\neq0$, a compact nonsingular flow $(N,Y)$
can be embedded into $\mathrm{HomWell}_{k}(\mathbb{R}^{m},V)$ for
some $m$ and $V$ iff there is a strongly adapted 1-form $\theta$
and a function $R>0$ smooth on $N$ such that 
\[
k\mathcal{L}_{Y}\theta+\left(1-\frac{k}{2}\right)d\left(\theta\cdot Y\right)-d\left(YY\left(\frac{R^{2}}{2}\right)\right)=0
\]
\end{thm}

\begin{proof}
Necessity is obvious from (\ref{eq:necessary}).

To prove sufficiency, we redo the proof of \Thmref{NLW_embed} without
mentioning $x$ or $G$. There is a Riemannian metric $g$ on $N$
such that $g\cdot Y=\theta$. Let $L$ be smooth on $N$ such that
$dL=\mathcal{L}_{Y}(\theta)$. Then we have an analogue of (\ref{eq:EulersetG}):

\begin{equation}
kL+\left(1-\frac{k}{2}\right)\left(\theta\cdot Y\right)-YY\left(\frac{R^{2}}{2}\right)=\text{const}\label{eq:EulersetG-1}
\end{equation}
Define another Riemannian metric $h$ on $N$ by $h_{\alpha\beta}=\frac{g_{\alpha\beta}-(\partial_{\alpha}R)(\partial_{\beta}R)}{R^{2}}$.
(WLOG, by adding a constant to $R^{2}$,$\left(g_{\alpha\beta}-(\partial_{\alpha}R)(\partial_{\beta}R)\right)_{\alpha,\beta}>0$)

Then by Nash embedding, and rescaling $g,h$ and $\theta$, we obtain
an isometric embedding $S:(N,h)\to S^{m-1}$ for some $m$. Then $\left\langle \partial_{\alpha}S,\partial_{\beta}S\right\rangle =h_{\alpha\beta}=\frac{g_{\alpha\beta}-(\partial_{\alpha}R)(\partial_{\beta}R)}{R^{2}}$. 

Let $Q=RS$. Then $\left\langle \partial_{\alpha}Q,\partial_{\beta}Q\right\rangle =\left(\partial_{\alpha}R\right)\left(\partial_{\beta}R\right)+R^{2}\left\langle \partial_{\alpha}S,\partial_{\beta}S\right\rangle =g_{\alpha\beta}$
and $Q:(N,g)\to\mathbb{R}^{m}$ is also an isometric embedding. This
means for any $W\in\mathfrak{X}(N\times\mathbb{T}^{d})$ we have $\theta(W)=g(Y,W)=\left\langle YQ,WQ\right\rangle $.
Then we have 
\begin{align*}
 & W\left(kL+\left(1-\frac{k}{2}\right)\left(\theta\cdot Y\right)-YY\left(\frac{|Q|^{2}}{2}\right)\right)=0\\
\iff & W\left\langle -YYQ,Q\right\rangle =k\left\langle -YYQ,WQ\right\rangle 
\end{align*}
The rest is the same as in \Thmref{NLW_embed}.
\end{proof}
\begin{rem}
When $k=2$, we recover (\ref{eq:new-condition}), and by \Exaref{torus},
this condition is not superfluous. 

When $k\notin\left\{ 0,2\right\} $, it is possible to set $r=1,\;\theta_{1}=f^{\frac{k+2}{k-2}},\theta_{2}=\text{const}$,
which then gives $2k\mathcal{L}_{Y}(\theta)=(k-2)d(\theta\cdot Y)$.
So the example doesn't help, and we do not know whether this condition
is superfluous. It is not known whether this is a distinct class of
flows from the one described in \Thmref{homo2con} (see \Queref{how_many}).
\end{rem}

\begin{rem}
The case $k=0$ is a bit special. (\ref{eq:necessary}) forces $0=-YY\left(\frac{R^{2}}{2}\right)+\theta\cdot Y$.
So $Y\left(\frac{R^{2}}{2}\right)$ is strictly increasing along the
flow of $Y$, and the rate of increase has a minimum positive value
(as $\theta\cdot Y>0$), so it will blow up, which cannot happen on
a compact space. Therefore we have the following theorem:
\end{rem}

\begin{thm}
A compact nonsingular flow $(N,Y)$ can never be embedded into $\mathrm{HomWell}_{0}(\mathbb{R}^{m},V)$
for any $m$ or V.
\end{thm}

\begin{cor}[Spherical embedding]
\label{cor:spherical} ~Let $k\neq0$ and $(N,Y)$ be a compact
nonsingular flow. There is a spherical embedding from $(N,Y)$ into
$\text{\ensuremath{\mathrm{HomWell}}}_{k}(\mathbb{R}^{n},V)$ if and
only if $(N,Y)$ has a strongly adapted 1-form $\theta$ where 
\[
2k\mathcal{L}_{Y}(\theta)=(k-2)d(\theta\cdot Y)
\]
\end{cor}

\begin{rem}
We also see that the condition $\mathcal{L}_{Y}(\theta)=\frac{k-2}{2k}d(\theta\cdot Y)$
implies 
\[
Y(\theta\cdot Y)=\mathcal{L}_{Y}(\theta)\cdot Y=\frac{k-2}{2k}Y(\theta\cdot Y)
\]
If $k\neq-2$, this means that the kinetic energy $\theta\cdot Y$
and the Lagrangian $L$ are constant along the flow of $Y.$
\end{rem}

\begin{cor}
\label{cor:0L} If $(N,Y)$ has a strongly adapted 1-form $\theta$
and $\mathcal{L}_{Y}(\theta)=0$, we can spherically embed $(N,Y)$
into $\text{\ensuremath{\mathrm{HomWell}}}_{2}(\mathbb{R}^{n},V)$
for some $n$ and $V$. 
\end{cor}

Again, by \Exaref{torus}, there are flows which can be embedded into
$\text{Well}(\mathbb{R}^{n},V),$ but not $\text{\ensuremath{\mathrm{HomWell}}}_{2}(\mathbb{R}^{n},V)$.
\begin{example}
Let $(N,Y)$ be a compact non-singular smooth flow.
\begin{itemize}
\item When $N=(S^{1})^{n}\subset\mathbb{C}^{n}=\mathbb{R}^{2n},\;Y(a)=ia,$
we can obviously embed $(N,Y)$ into $\text{HomWell}_{2}(\mathbb{R}^{2n},V)$
where $V(x)=\frac{|x|^{2}}{2}$ (actually smooth at the origin). With
the induced Euclidean metric on $N$, $\nabla_{Y}Y=0.$ Let $\theta=Y^{b}$.
Then $\forall Z\in\mathfrak{X}N:$ 
\begin{align*}
\mathcal{L}_{Y}(\theta)\cdot Z & =Y(\theta\cdot Z)-\theta\cdot[Y,Z]=Y\left\langle Y,Z\right\rangle -\left\langle Y,[Y,Z]\right\rangle \\
 & =\left\langle Y,\nabla_{Y}Z\right\rangle -\left\langle Y,\nabla_{Y}Z-\nabla_{Z}Y\right\rangle \\
 & =\left\langle Y,\nabla_{Z}Y\right\rangle =\frac{1}{2}Z(|Y|^{2})=0
\end{align*}
so $\mathcal{L}_{Y}(\theta)=0$, $Q(a)=a,P=YQ=iQ,YP=iP=-Q=-\nabla_{\mathbb{R}^{2n}}V(Q)$.
We note that $(N,Y)$ is both \emph{isometric} and \emph{strongly
geodesible }(see below).
\item When $(N,Y)$ is isometric / Killing, i.e. there is a Riemannian metric
$g$ such that $\mathcal{L}_{Y}g=0$: By letting $\theta=g\cdot Y$,
we have $\mathcal{L}_{Y}(\theta)=\mathcal{L}_{Y}(g)\cdot Y+g\cdot\mathcal{L}_{Y}(Y)=0$
and \Corref{0L} applies.
\item When $(N,Y)$ is strongly geodesible, i.e. there is a Riemannian metric
$g$ such that $\nabla_{Y}Y=0$ and $|Y|_{g}=1$: let $\theta=g\cdot Y$
and we have 
\begin{align*}
\mathcal{L}_{Y}(\theta)\cdot Z & =Yg(Y,Z)-g(Y,[Y,Z])=g(Y,\nabla_{Y}Z)-g(Y,\nabla_{Y}Z-\nabla_{Z}Y)\\
 & =g(Y,\nabla_{Z}Y)=\frac{1}{2}Z(|Y|_{g}^{2})=0
\end{align*}
So \Corref{0L} applies.
\item When $N=M\times[0,1]/\left((y,1)\sim(\phi(y),0)\right)=\left\{ [(y,t)]:y\in M,t\in[0,1]\right\} $
where $\phi:M\to M$ is a diffeomorphism on a compact smooth manifold
$M$, and $Y=(X,\partial_{t})$ where $X\in\mathfrak{X}M$ such that
$X=X\circ\phi$: Let $\theta=(0,dt)$, then $\mathcal{L}_{Y}(\theta)=0$.
This is the relevant case for universal Turing machines, as shown
in \parencite{taoUniversalityPotentialWell2017}.
\end{itemize}
\end{example}

\printbibliography
\end{document}